\documentclass[11pt,reqno]{amsart}
\usepackage{amsmath,amssymb,amsthm,mathtools,xcolor}
\usepackage[letterpaper,margin=1.13in]{geometry}
\usepackage[hidelinks]{hyperref}
\usepackage{microtype}
\newif\ifshowchanges
\showchangestrue

\newtheorem{theorem}{Theorem}[section]
\newtheorem{proposition}[theorem]{Proposition}
\newtheorem{lemma}[theorem]{Lemma}
\newtheorem{corollary}[theorem]{Corollary}
\theoremstyle{definition}
\newtheorem{example}[theorem]{Example}

\newtheorem*{remark*}{Remark}
\numberwithin{equation}{section}
\DeclareMathOperator{\Tr}{Tr}
\DeclareMathOperator{\spec}{spec}
\DeclareMathOperator{\Sym}{Sym}
\DeclareMathOperator{\arcosh}{arcosh}
\newcommand{\Pn}{\mathbb P_n}
\newcommand{\logmaj}{\prec_{\log}}
\newcommand{\wlogmaj}{\prec_{w\log}}
\newcommand{\el}{\leq_{\lambda}}
\newcommand{\eqs}{=_{\lambda}}
\newcommand{\near}{\mathrel{\preceq}}
\newcommand{\geo}{\mathbin{\sharp}}
\newcommand{\ns}{\mathbin{\widetilde{\natural}}}
\newcommand{\norm}[1]{\lVert #1\rVert}
\newcommand{\fnorm}[1]{\lVert #1\rVert_F}
\newcommand{\ind}[1]{\mathbf 1_{\{#1\}}}
\newcommand{\wh}{\bigwedge^k}

\title[Weighted convex order for near spectral curves]{Weighted Convex Order for Near Spectral Geometric Curves}

\author{Muritala Abiodun Ibrahim}
\address{Department of Mathematics and Statistics, Auburn University, Auburn, Alabama 36849--5310, USA}
\email{mai0015@auburn.edu}

\author{Trung Hoa Dinh}
\address{Department of Mathematics and Statistics, Troy University, Troy, Alabama 36082, USA}
\email{thdinh@troy.edu}
\address{Van Lang University, Ho Chi Minh City, Vietnam}
\email{hoa.dinhtrung@vlu.edu.vn}

\author{Huajun Huang}
\address{Department of Mathematics and Statistics, Auburn University, Auburn, Alabama 36849--5310, USA}
\email{huanghu@auburn.edu}


\begin{document} \setlength{\baselineskip}{16.2pt}
\begin{abstract}
The geometric mean and the near geometric mean of $A^{-1}$ and $B$ give positive matrices $M$ and $N$, respectively. We compare the curves $P_t=M^tAM^t$ and $Q_t=N^tAN^t$ for positive definite matrices $A,B$. We prove that $Q_t\logmaj P_t$ for $0<t<1/2$ and that the order reverses for $t<0$ and $t>1/2$. The midpoints are isospectral. For every $t\notin\{0,1/2\}$, equality of the traces is equivalent to $AB=BA$. The proof begins with the weighted convex order $\Tr(Af(M))\leq\Tr(Af(N))$ for scalar convex $f$. We give an exact formula for the difference, characterize equality, and construct a positive second-order remainder measure whose total mass is an explicit quadratic expression in $[A,B^{1/2}]$. For each $r>1$, we also show that the midpoint is the only positive parameter at which $Q_t(A,B)^r\logmaj Q_t(A^r,B^r)$ holds in every dimension.
\end{abstract}
\maketitle

\section{Introduction}
Let $A,B$ be positive definite $n\times n$ matrices. The unique positive definite solution of $MAM=B$ is the geometric mean $M=A^{-1}\geo B$. The spectral geometric curve $P_t=M^tAM^t$, $t\in\mathbb R$, satisfies $P_0=A$ and $P_1=B$. Its midpoint $P_{1/2}$ is the spectral geometric mean of Fiedler and Pt\'ak \cite{FP}. If we use the near geometric mean of Franco, Gan, and Kim \cite{FGK} in place of $\geo$, we obtain $N=A^{-1}\star B$ and the near spectral geometric curve $Q_t=N^tAN^t$. The two curves agree at $t=0$, although $Q_1$ can differ from $B$. For a given $t$, how do the eigenvalues of $Q_t$ compare with those of $P_t$?

The two constructions differ only in the order of two factors. With $Z=A^{1/2}B^{1/2}$, congruence by $A^{1/2}$ gives
\begin{equation}\label{eq:polar}
 A^{1/2}MA^{1/2}=(ZZ^*)^{1/2},\qquad A^{1/2}NA^{1/2}=(Z^*Z)^{1/2}.
\end{equation}
The midpoint $P_{1/2}=M^{1/2}AM^{1/2}$ is similar to $AM$, and hence to $A^{1/2}MA^{1/2}$; likewise $Q_{1/2}$ is similar to $A^{1/2}NA^{1/2}$. Thus both midpoints have the singular values of $Z$ as their eigenvalues. Taking determinants in \eqref{eq:polar} gives $\det M=\det N$, and therefore
\[
 \det P_t=\det Q_t=(\det A)^{1-t}(\det B)^t\qquad(t\in\mathbb R).
\]
Our main result, Theorem~\ref{thm:logmaj}, compares the products of the $k$ largest eigenvalues of $P_t$ and $Q_t$ away from the common starting point $t=0$ and the isospectral midpoint $t=1/2$. On the three intervals determined by these points,
\begin{equation}\label{eq:introlog}
\begin{aligned}
 Q_t&\logmaj P_t &&(0<t<1/2),\\
 P_t&\logmaj Q_t &&(t<0\ \text{or}\ t>1/2).
\end{aligned}
\end{equation}
If $A$ and $B$ commute, then $M=N$ and the two curves coincide. Otherwise, for every $t\notin\{0,1/2\}$ at least one partial-product inequality in \eqref{eq:introlog} is strict, and so the traces of $P_t$ and $Q_t$ are strictly ordered. In particular, $B=P_1\logmaj Q_1$, and the endpoint spectra coincide if and only if $A$ and $B$ commute. This strictness is sharp for individual products: for any fixed $t\notin\{0,1/2\}$ and any $1\leq k<n$, there are pairs for which only the $k$th partial-product inequality is strict.

The proof rests on a weighted convex order for $M$ and $N$. Theorem~\ref{thm:weighted} states that for every convex function $f$ on an open interval containing the spectra of $M$ and $N$,
\begin{equation}\label{eq:introweight}
 \Tr(Af(M))\leq\Tr(Af(N)).
\end{equation}
For $X\in\{M,N\}$, let $E_\lambda^X$ be the spectral projection of $X$ at $\lambda$. Then $\Tr(Af(X))/\Tr(A)$ is the expectation of $f$ under the probability measure that assigns mass $\Tr(AE_\lambda^X)/\Tr(A)$ to each eigenvalue $\lambda$. By \eqref{eq:polar}, $\Tr(AM)$ and $\Tr(AN)$ both equal the sum of the singular values of $Z$. So the two measures have the same mean, and \eqref{eq:introweight} says that they are ordered in the convex order. Franco, Gan, and Kim proved the comparison $M\logmaj N$ \cite[Corollary~5.5]{FGK}, which involves only the eigenvalues. The weights in \eqref{eq:introweight} also record where the eigenspaces sit relative to $A$.

Our approach is partly motivated by earlier work on the corrected Bessis--Moussa--Villani (BMV) conjecture. The pinching refinement proposed in \cite{DinhBMV} led to the positive second-order measure representation established in \cite{DLVBMV}, in which the mass of the measure is determined by a quadratic pinching remainder. Here we develop a corresponding representation for weighted spectral measures. The argument is different: the matrix weight $A$ is not introduced by a trace-preserving conditional expectation, and the required positivity follows from an explicit Bregman remainder identity. In particular, \eqref{eq:introweight} holds for every scalar convex function, without an operator convexity assumption.

To prove \eqref{eq:introweight}, we write the trace gap exactly as a sum of nonnegative terms. The identity also gives the equality cases. Suppose $A$ and $B$ have no common nonzero proper reducing subspace. Then equality holds if and only if $f$ is affine on $[\lambda_n(N),\lambda_1(N)]$. In general, the criterion applies on each common irreducible summand: $f$ must be affine on the corresponding spectral interval of $N$.

For the curves themselves, take $f(x)=x^{2t}$ in \eqref{eq:introweight}. By cyclicity of the trace, $\Tr P_t=\Tr(AM^{2t})$ and $\Tr Q_t=\Tr(AN^{2t})$. The function $x^{2t}$ is concave for $0\leq t\leq1/2$ and convex otherwise, and this change accounts for both reversals in \eqref{eq:introlog}. More precisely, there is a nonnegative piecewise linear function $H_{A,B}$ such that
\begin{equation}\label{eq:introrem}
 \Tr(Q_t-P_t)=2t(2t-1)\int_0^\infty c^{2t-2}H_{A,B}(c)\,dc\qquad(t\in\mathbb R),
\end{equation}
and
\[
 \int_0^\infty H_{A,B}(c)\,dc=\tfrac12\Tr(Q_1-B)
 =\tfrac14\fnorm{A^{-1/2}[A,B^{1/2}]A^{-1/2}}^2.
\]
Hence, for noncommuting $A$ and $B$, the trace difference vanishes only at $t=0$ and $t=1/2$, and it changes sign with nonzero slope at both points. To pass from traces to the product of the $k$ largest eigenvalues, we apply the trace inequality to the symmetric powers of the $k$th exterior powers of $A$ and $B$, and take a limit.

We also ask how the curve $Q_t=Q_t(A,B)$ behaves under powers of its inputs. Fix $r>1$. For which $t>0$ does
\[
 Q_t(A,B)^r\logmaj Q_t(A^r,B^r)
\]
hold for every pair $A,B$? Theorem~\ref{thm:power} shows that this happens only for $t=1/2$, whatever the value of $r>1$. At the midpoint, the inequality follows from Araki's inequality. For $0<t<1/2$, a comparison with the log-Euclidean mean gives counterexamples in dimension two. For $t>1/2$, a three-dimensional construction separates the growth rates of the largest eigenvalues of the two sides. Thus the result is a characterization for each fixed exponent $r>1$, rather than just a failure of a proposed inequality at isolated parameter values.

Section~\ref{sec:prelim} starts with the near geometric mean and its comparison with the arithmetic mean. Its change of endpoint gives the near spectral curve in Section~\ref{sec:near}. We prove the weighted convex order in Section~\ref{sec:weighted} and the log-majorization reversal in Section~\ref{sec:log}. Section~\ref{sec:power} treats the log-Euclidean mean and powers of the inputs.

\section{The near geometric curve}\label{sec:prelim}
Write $\Pn$ for the cone of complex Hermitian positive definite $n\times n$ matrices. For $A,B\in\Pn$ and $t\in\mathbb R$, put
\[
 V=A^{-1/2}BA^{-1/2},\qquad Y=B^{1/2}A^{-1}B^{1/2}.
\]
The geometric and near geometric curves are, respectively,
\[
 A\geo_t B=A^{1/2}V^tA^{1/2},\qquad A\star_t B=A^{1/2}Y^tA^{1/2}.
\]
We omit the weight for $t=1/2$. For the geometric mean and its Riccati equation see \cite{PW,KA,BhatiaPD}. Franco, Gan, and Kim \cite{FGK} introduced the near geometric mean for $0\leq t\leq1$. All powers here are defined by the spectral calculus, so both formulas extend to real $t$.

The matrices $V$ and $Y$ are $ZZ^*$ and $Z^*Z$ for $Z=A^{-1/2}B^{1/2}$. In particular they have the same eigenvalues. At $t=1$, congruence by $A^{1/2}$ takes $V$ to $B$ and $Y$ to $C=A^{1/2}YA^{1/2}$. This change of endpoint gives the identity \cite{FGK}
\begin{equation}\label{eq:change}
 A\star_t B=A\geo_t C,\qquad C=A\star_1 B=A^{1/2}YA^{1/2}.
\end{equation}
For each $A$ this change of variable is a bijection. Given $C$, the equation $WA^{-1}W=A^{-1/2}CA^{-1/2}$ has the unique positive solution $W=A\geo(A^{-1/2}CA^{-1/2})$. Thus its inverse is
\begin{equation}\label{eq:inverse}
 B=\bigl[A\geo(A^{-1/2}CA^{-1/2})\bigr]^2.
\end{equation}
Also, for any $t>0$, $A\star_t B\geq A$ if and only if $Y^t\geq I$, or equivalently $A\leq B$. For $t<0$, the same equivalence reads $A\star_t B\leq A$ if and only if $A\leq B$.

We shall also use the identities
\[
 (aA)\star_t(bB)=a^{1-t}b^t A\star_t B,\quad A^{-1}\star_t B^{-1}=(A\star_t B)^{-1},\quad
 \det(A\star_t B)=(\det A)^{1-t}(\det B)^t
\]
for all $a,b>0$ and $t\in\mathbb R$. The defining formula also gives unitary covariance. When $A$ and $B$ commute, $A\star_t B=A^{1-t}B^t$. Franco, Gan, and Kim \cite{FGK} obtained these identities on $[0,1]$, together with the Lie--Trotter formula
\begin{equation}\label{eq:nearLT}
 \lim_{r\to0}(A^r\star_t B^r)^{1/r}=\exp\bigl((1-t)\log A+t\log B\bigr).
\end{equation}
For any fixed real $t$ the same limit follows by expanding at $r=0$. Indeed,
\[
 B^{r/2}A^{-r}B^{r/2}=I+r(\log B-\log A)+O(r^2),
\]
whence $A^r\star_t B^r=I+r((1-t)\log A+t\log B)+O(r^2)$.

For comparison, write $M=A^{-1}\geo B$ and $A\natural_t B=M^tAM^t$. The two classical curves satisfy
\begin{equation}\label{eq:georeversal}
 A\geo_t B=B\geo_{1-t}A,
\end{equation}
\begin{equation}\label{eq:specreversal}
 A\natural_t B=B\natural_{1-t}A\qquad(t\in\mathbb R).
\end{equation}
For the first identity, congruence invariance reduces the assertion to $I\geo_t X=X\geo_{1-t}I=X^t$. For the second, $MAM=B$ gives $B^{-1}\geo A=M^{-1}$, and hence $B\natural_{1-t}A=M^{t-1}BM^{t-1}=M^tAM^t$. For $\star_t$, reversal can fail already at $t=1$. Indeed, $A\star_1 B=B$ is equivalent to $Z^*Z=ZZ^*$ for $Z=A^{-1/2}B^{1/2}$. This $Z$ is similar to the positive matrix $A^{-1/4}B^{1/2}A^{-1/4}$. Thus it is normal exactly when it is Hermitian, which is equivalent to $AB=BA$.

\begin{proposition}\label{prop:starcomm}
For any distinct real $s,t$, $A\star_s B$ and $A\star_t B$ commute if and only if $A$ and $B$ commute. In this case all members of the curve commute.
\end{proposition}
\begin{proof}
Commutativity of $A\star_s B$ and $A\star_t B$ is equivalent to $Y^sAY^t=Y^tAY^s$, hence to $AY^{t-s}=Y^{t-s}A$. Since $s\ne t$, this gives $AY=YA$. The positive Riccati equation $WA^{-1}W=Y$ then has the solution $(AY)^{1/2}$. Uniqueness gives $B^{1/2}=(AY)^{1/2}$, so $AB=BA$. The converse follows from $A\star_t B=A^{1-t}B^t$ for commuting inputs.
\end{proof}

\subsection{Comparison with the arithmetic mean}
Franco, Gan, and Kim \cite{FGK} asked how $A\star_t B$ compares with the arithmetic mean.

Throughout, eigenvalues are listed in nonincreasing order, with multiplicity. We write $X\el Y$ for componentwise inequality and $X\eqs Y$ for equality of the eigenvalue lists. Write $X\wlogmaj Y$ if
\[
 \prod_{j=1}^k\lambda_j^\downarrow(X)\leq\prod_{j=1}^k\lambda_j^\downarrow(Y)\quad (1\leq k\leq n).
\]
Adding equality at $k=n$ gives $X\logmaj Y$. For $V$ and $Y$ defined above we have
\[
 A^{-1/2}[(1-t)A+tB]A^{-1/2}=(1-t)I+tV,\qquad A^{-1/2}(A\star_tB)A^{-1/2}=Y^t.
\]
Since $V\eqs Y$, the scalar inequality $x^t\leq1-t+tx$ for $0\leq t\leq1$ gives
\begin{equation}\label{eq:normalized}
 A^{-1/2}(A\star_tB)A^{-1/2}\eqs V^t\el A^{-1/2}[(1-t)A+tB]A^{-1/2}\qquad(0\leq t\leq1).
\end{equation}
Put $\alpha=\lambda_{\min}(V)$ and $\beta=\lambda_{\max}(V)$. A sufficient condition for $A\star_tB\leq(1-t)A+tB$ at $t\in[0,1]$ is
\begin{equation}\label{eq:arithcondition}
 \beta^t\leq1-t+t\alpha.
\end{equation} Indeed, $Y^t\leq\beta^t I$ and $(1-t+t\alpha)I\leq(1-t)I+tV$. The function $s\mapsto1-s+s\alpha-\beta^s$ is concave and vanishes at $0$, so the condition at $t$ implies the same condition at every $s\in[0,t]$. The interval on which $A\star_tB\leq(1-t)A+tB$ also depends on the relative eigenspaces of $V$ and $Y$.

\begin{proposition}\label{prop:arithinterval}
The set of $t\in[0,1]$ for which $A\star_tB\leq(1-t)A+tB$ is an interval $[0,t_*]$. Its right endpoint satisfies $t_*>0$ if and only if
\[
 E:=V-I-\log Y\geq0\qquad\text{and}\qquad\ker E\subseteq\ker(\log Y).
\]
\end{proposition}
\begin{proof}
For $t>0$ the inequality is equivalent to $R_t\leq E$, where $R_t=(Y^t-I)/t-\log Y$. For each $y>0$, convexity of $t\mapsto y^t$ gives $0\leq(y^s-1)/s-\log y\leq(y^t-1)/t-\log y$ for $0<s\leq t$. For $s>0$ the first inequality is strict unless $y=1$. Thus $0\leq R_s\leq R_t$, $R_t\to0$ as $t\downarrow0$, and $\ker R_t=\ker(\log Y)$. Continuity of $[(1-t)A+tB]-(A\star_tB)$ gives the closed interval assertion. If $R_t\leq E$ for some $t>0$, positivity forces $E\geq0$ and $\ker E\subseteq\ker R_t$. Conversely, under the two stated conditions each $R_t$ vanishes on $\ker E$. On its orthogonal complement $E$ is positive definite, so $R_t\leq E$ for all sufficiently small $t>0$. If $E=0$, the kernel condition gives $R_t=0$ for every $t>0$.
\end{proof}

\begin{example}\label{ex:arith}
Set $t=2/5$ and take
\[
 A=\begin{pmatrix}1&-2\\-2&5\end{pmatrix},\qquad B=\begin{pmatrix}1&-1\\-1&5\end{pmatrix}.
\]
The eigenvalues of $(1-t)A+tB$ are $3\pm2\sqrt{41}/5$. For $A\star_tB$, the matrix
\[
 Y=B^{1/2}A^{-1}B^{1/2}=\frac15\begin{pmatrix}17&11\\11&13\end{pmatrix}
\]
has eigenvalues $y_\pm=3\pm\sqrt5$. Interpolating $Y^t$ at these two points and using $\Tr(AY)=38/5$ gives
\[
 \Tr(A\star_tB)=\left(3-\frac{26}{5\sqrt5}\right)y_+^t+\left(3+\frac{26}{5\sqrt5}\right)y_-^t.
\]
Together with $\det(A\star_tB)=4^{2/5}$ this gives
\[
 5.788<\lambda_{\max}(A\star_tB)<5.790,\qquad \lambda_{\max}((1-t)A+tB)<5.562.
\]
On the other hand, $\det(A\star_tB)=4^{2/5}<61/25=\det((1-t)A+tB)$. The two weak log-majorizations therefore both fail.
\end{example}

For ordered inputs the failure can be arbitrarily large, at every interior parameter.
\begin{proposition}\label{prop:arithunbounded}
There is a fixed $B\in\mathbb P_2$ such that, for every $0<t<1$, there are matrices $A_v\in\mathbb P_2$, $0<A_v\leq B$, such that
\[
 \frac{\lambda_{\max}(A_v\star_tB)}{\lambda_{\max}((1-t)A_v+tB)}\longrightarrow\infty\quad(v\to\infty).
\]
For all sufficiently large $v$, the two matrices are incomparable under weak log-majorization.
\end{proposition}
\begin{proof}
Take
\[
 T=\begin{pmatrix}1&1\\1&2\end{pmatrix},\qquad B=T^2,\qquad A_v=T\operatorname{diag}(1,v^{-1})T,\qquad v>1.
\]
Then $A_v\leq B$ and $Y=TA_v^{-1}T=\operatorname{diag}(1,v)$. The matrix $A_v\star_tB$ has the same eigenvalues as $Y^{t/2}A_vY^{t/2}$. Its largest eigenvalue is therefore at least the $(2,2)$ entry of the latter matrix, namely $v^t(1+4/v)$. The arithmetic mean is bounded above by $B$. On the other hand, $\det T=1$, so the strict scalar arithmetic--geometric mean inequality gives
\[
 \det(A_v\star_tB)=v^{t-1}<(1-t)v^{-1}+t=\det((1-t)A_v+tB).
\]
The largest eigenvalues and determinants thus have opposite orders for large $v$. For example, $t=1/3$ and $v=512$ suffice, since $\lambda_{\max}(B)<\Tr B=7$.
\end{proof}

\subsection{Comparison with other means}
We compare $B^{-1}\star_t A$ with the $F$-mean and a R\'enyi mean. For the spectral geometric and R\'enyi means, see also \cite{DFrenyi}. For $0<t\leq1$ write
\[
 F_t(A,B)=(A^{-1}\geo_t B)^{1/2}A^{2-2t}(A^{-1}\geo_t B)^{1/2},\qquad
 R_t(A,B)=\bigl(A^{(1-t)/(2t)}BA^{(1-t)/(2t)}\bigr)^t.
\]
These are the $F$-mean of Dinh, Tam, and Vuong \cite{DTV} and a special case of the R\'enyi mean \cite{JKT}. The weighted Bures--Wasserstein mean is $A\diamond_t B=((1-t)I+tM)A((1-t)I+tM)$, where $M=A^{-1}\geo B$.
\begin{proposition}\label{prop:othermeans}
If $A\leq B$, $I\leq B$, and $1/2\leq t\leq1$, then
\[
 \begin{aligned}
 B^{-1}\star_t A&\el F_t(A,B)\logmaj R_t(A,B)\logmaj A\natural_t B\\
 &\wlogmaj A\diamond_t B\leq(1-t)A+tB.
 \end{aligned}
\]
\end{proposition}
\begin{proof}
Put $H=A^{1/2}BA^{1/2}$. Interchanging $Z^*Z$ and $ZZ^*$ gives
\[
 B^{-1}\star_t A\eqs H^{t/2}B^{-1}H^{t/2},\qquad F_t(A,B)\eqs H^{t/2}A^{1-2t}H^{t/2}.
\]
For the second identity, first pass from $F_t(A,B)$ to $A^{1-t}(A^{-1}\geo_t B)A^{1-t}=A^{1/2-t}H^t A^{1/2-t}$. The L\"owner--Heinz inequality gives $A^{2t-1}\leq B^{2t-1}\leq B$, so $B^{-1}\leq A^{1-2t}$. Congruence by $H^{t/2}$ and eigenvalue monotonicity give the first comparison. The three middle comparisons are due to Jeong, Kim, and Tam \cite[Corollary~3.9]{JKT}. The last follows from
\[
 (1-t)A+tB-A\diamond_t B=t(1-t)(I-M)A(I-M)\geq0
\]
(see also \cite{BJL}).
\end{proof}

\section{The near spectral geometric curve}\label{sec:near}
\subsection{Construction from the near geometric mean}
Apply \eqref{eq:change} with first input $A^{-1}$. The resulting curve $C_t=A^{-1}\star_t B$ has endpoint
\begin{equation}\label{eq:D}
 D=C_1=A^{-1/2}B^{1/2}AB^{1/2}A^{-1/2}
\end{equation}
and midpoint $N=C_{1/2}=A^{-1}\geo D$. Thus $N$ is the positive solution of $NAN=D$. Just as $M=A^{-1}\geo B$ gives the spectral geometric curve $P_t=A\natural_t B=M^tAM^t$ from $A$ to $B$, $N$ gives the near spectral geometric curve
\begin{equation}\label{eq:Q}
 Q_t=A\ns_t B:=N^tAN^t,\qquad N=A^{-1}\star B,\quad t\in\mathbb R.
\end{equation}
In particular $Q_0=A$ and $Q_1=D$. The two endpoint substitutions are
\begin{equation}\label{eq:sub}
 C_t=A^{-1}\geo_t D,\qquad Q_t=A\natural_t D.
\end{equation}
For fixed $A$, every $D\in\Pn$ occurs for precisely one $B$. In fact, \eqref{eq:inverse} with $A^{-1}$ in place of $A$ gives
\[
 B=\bigl[A^{-1}\geo(A^{1/2}DA^{1/2})\bigr]^2.
\]
We write $\norm{\cdot}$ for the operator norm, $\fnorm{\cdot}$ for the Frobenius norm, and $[X,Y]=XY-YX$.

For commuting $A$ and $B$ one has $Q_t=A^{1-t}B^t$. For arbitrary inputs the curve is unitarily covariant. Replacing $(A,B)$ by $(aA,bB)$ multiplies $Q_t$ by $a^{1-t}b^t$, and replacing the inputs by their inverses gives $Q_t^{-1}$. These facts follow from the corresponding identities for $\star$. In particular, $\det Q_t=(\det A)^{1-t}(\det B)^t$.

Put $T=B^{1/2}$. The Riccati equation $MAM=T^2$ allows us to express $N$ in terms of $M$ and $T$.

\begin{lemma}\label{lem:transport}
The matrices $M$ and $N$ satisfy
\begin{equation}\label{eq:transport}
 MAM=T^2,\qquad A^{1/2}NA^{1/2}=TM^{-1}T.
\end{equation}
Moreover, $A^{1/2}MA^{1/2}\eqs A^{1/2}NA^{1/2}$, $\Tr(AM)=\Tr(AN)$, and $\det M=\det N$.
\end{lemma}
\begin{proof}
By definition $A^{1/2}MA^{1/2}=(A^{1/2}BA^{1/2})^{1/2}$. Squaring gives $MAM=T^2$. The positive matrix $TM^{-1}T$ has square $TM^{-1}T^2M^{-1}T=TAT$, so it equals $A^{1/2}NA^{1/2}$. Put $Z=A^{1/2}B^{1/2}$. Then
\[
 A^{1/2}MA^{1/2}=(ZZ^*)^{1/2},\qquad A^{1/2}NA^{1/2}=(Z^*Z)^{1/2}.
\]
Since $Z$ is invertible, the two matrices under the square roots are unitarily similar by the polar decomposition. Their positive square roots therefore have the same trace and determinant. Cyclicity yields $\Tr(AM)=\Tr(AN)$, and cancellation of $\det A$ gives $\det M=\det N$.
\end{proof}

The midpoints $P_{1/2}$ and $Q_{1/2}$ are isospectral. Indeed, $M^{1/2}AM^{1/2}$ is isospectral with $A^{1/2}MA^{1/2}$, and the same holds with $N$ in place of $M$. By Lemma~\ref{lem:transport}, their common eigenvalues are the singular values of $A^{1/2}B^{1/2}$. They are also the eigenvalues of the principal square roots $(AB)^{1/2}$ and $(AD)^{1/2}$, since $AB$ and $AD$ are similar to $A^{1/2}BA^{1/2}$ and $B^{1/2}AB^{1/2}$, respectively.

\begin{proposition}\label{prop:endpoint}
One has $B\logmaj D$ and
\begin{equation}\label{eq:endpoint}
 \Tr(D-B)=\tfrac12\fnorm{A^{-1/2}[A,B^{1/2}]A^{-1/2}}^2.
\end{equation}
In particular, $D=B$, equality of the endpoint spectra, and equality of the endpoint traces are each equivalent to $AB=BA$.
\end{proposition}
\begin{proof}
Let $Z=A^{1/2}TA^{-1/2}$. Then $Z^*Z=D$, and $Z$ is similar to $T$. Weyl's majorant theorem \cite{BhatiaMA} compares the products of the largest moduli of the eigenvalues of $Z$ with the corresponding products of its singular values. Squaring gives $B\logmaj D$. Also $\Tr Z^2=\Tr B$, so
\[
 2\Tr(D-B)=\fnorm{Z-Z^*}^2,\qquad Z-Z^*=A^{-1/2}[A,T]A^{-1/2}.
\]
Equality of traces is therefore equivalent to $AT=TA$, or $AB=BA$. Commuting inputs give $D=B$, which implies equality of spectra and traces.
\end{proof}
Consequently, the bounds for the Frobenius norm under congruence by $A^{-1/2}$ give
\[
 \frac{\fnorm{[A,B^{1/2}]}^2}{2\lambda_{\max}(A)^2}
 \leq\Tr(D-B)\leq
 \frac{\fnorm{[A,B^{1/2}]}^2}{2\lambda_{\min}(A)^2}.
\]
\subsection{Interpolation and near order}
For $s,t\in\mathbb R$, $N^{t-s}$ is positive and $N^{t-s}Q_sN^{t-s}=Q_t$. The Riccati characterization gives
\begin{equation}\label{eq:transportst}
 Q_s^{-1}\geo Q_t=N^{t-s},\qquad s,t\in\mathbb R.
\end{equation}
In particular, $Q_t$ is the unique positive matrix $X$ for which $A^{-1}\geo X=N^t$. The near order of Dumitru and Franco \cite{DFnear} is $X\near Y$ if $X^{-1}\geo Y\geq I$. It satisfies
\begin{equation}\label{eq:nearorderchain}
 X\leq Y\Longrightarrow X\near Y\Longrightarrow X\el Y\Longrightarrow X\wlogmaj Y
\end{equation}
by monotonicity of $\geo$ and \cite{GH}. We use the following characterization of the near order due to Gan and Huang \cite[Corollary~2.6]{GH}.
\begin{lemma}\label{lem:near}
Let $A,P,Q\in\Pn$. If $P$ and $Q$ commute, then
\[
 P\leq Q\quad\Longleftrightarrow\quad PAP\near QAQ.
\]
\end{lemma}
\begin{proof}
Commutativity makes $R=QP^{-1}$ positive definite and gives $R(PAP)R=QAQ$. Uniqueness in the Riccati equation yields $(PAP)^{-1}\geo(QAQ)=R$. The near-order test is therefore $R\geq I$, which is equivalent to $Q\geq P$.
\end{proof}
Metric in-betweenness compares the distance from an endpoint to a point of an interpolating curve \cite{Aud, DDF}. For the function $d$ studied in \cite{DFnear}, the distance between two points of $Q_t$ is proportional to their parameter difference. This follows from \eqref{eq:transportst} and the identity $Q_t=A\natural_tD$; for the geometry of the spectral curve see \cite{LL,GT}.
\begin{proposition}\label{prop:interp}
For all real $s,t,u$ one has
\begin{equation}\label{eq:interp}
 Q_s\natural_uQ_t=Q_{(1-u)s+ut},\quad d(Q_s,Q_t)=|t-s|d(A,D),\quad d(X,Y)=2\norm{\log(X^{-1}\geo Y)}.
\end{equation}
For any $s<t$,
\[
 Q_s\near Q_t\Longleftrightarrow N\geq I\Longleftrightarrow A\near D,\qquad
 Q_t\near Q_s\Longleftrightarrow N\leq I\Longleftrightarrow D\near A.
\]
The sharp scalar bounds are
\begin{equation}\label{eq:nearsharp}
 \lambda_{\min}(N)^{2(t-s)}Q_s\near Q_t\near\lambda_{\max}(N)^{2(t-s)}Q_s.
\end{equation}
\end{proposition}
\begin{proof}
Inserting \eqref{eq:transportst} into the definitions of $\natural_u$ and $d$ gives \eqref{eq:interp}. For $a>0$, the matrices $a^{1/2}N^s$ and $N^t$ commute. Lemma~\ref{lem:near} gives
\[
\begin{aligned}
 aQ_s\near Q_t&\Longleftrightarrow a^{1/2}I\leq N^{t-s}\Longleftrightarrow a\leq\lambda_{\min}(N)^{2(t-s)},\\
 Q_t\near aQ_s&\Longleftrightarrow N^{t-s}\leq a^{1/2}I\Longleftrightarrow a\geq\lambda_{\max}(N)^{2(t-s)}.
\end{aligned}
\]
This proves the bounds and their sharpness. Taking $a=1$ gives the monotonicity assertions, and taking $s=0,t=1$ gives the endpoint tests.
\end{proof}
Thus two ordered points of the curve determine its near-order monotonicity on the whole real line. In terms of the original inputs, $N\geq I$ is equivalent to $(B^{1/2}AB^{1/2})^{1/2}\geq A$, and $N\leq I$ to the reverse inequality. If $N\leq I$, the ordered eigenvalues, and hence their partial products, are nonincreasing in $t$. Since $B\logmaj D$, this yields $B\logmaj D\wlogmaj Q_t\wlogmaj A$ for $0\leq t\leq1$. When $N\geq I$, each ordered eigenvalue of $Q_t$ is nondecreasing in $t$.

Setting $s=0$ in \eqref{eq:interp} gives $A\natural_uQ_t=Q_{ut}$ for every $u,t\in\mathbb R$. If the outer mean in this identity is replaced by the near spectral mean, commutativity of the original inputs is necessary as soon as both parameters are nonzero.
\begin{proposition}\label{prop:Qcomm}
The following conditions on $A,B\in\Pn$ are equivalent.
\begin{enumerate}
 \item $AB=BA$.
 \item $Q_s$ and $Q_t$ commute for some $s\ne t$.
 \item $A\ns_t Q_s=Q_{st}$ for some $s,t\ne0$.
\end{enumerate}
When these conditions hold, all pairs $Q_s,Q_t$ commute and the identity in the third condition holds for every $s,t\in\mathbb R$.
\end{proposition}
\begin{proof}
Suppose that $Q_s$ commutes with $Q_t$ for $s\ne t$. Their transport $Q_s^{-1}\geo Q_t=N^{t-s}$ then commutes with $Q_s$. Functional calculus shows that $N$ commutes with $Q_s$ and therefore with $A=N^{-s}Q_sN^{-s}$. Put $T=B^{1/2}$. Since $(TAT)^{1/2}=A^{1/2}NA^{1/2}=AN$, the positive matrices $T$ and $A^{1/2}N$ both solve $YAY=A^2N^2$. Uniqueness gives $T=A^{1/2}N$, so $A$ commutes with $B$.

Now suppose that $A\ns_tQ_s=Q_{st}$ for $s,t\ne0$. The positive transports from $A$ to these matrices must agree. Thus $(A^{-1}\star Q_s)^t=N^{st}$, and taking the $1/t$ power gives $A^{-1}\star Q_s=N^s=A^{-1}\geo Q_s$. Conjugating this last equality by $A^{1/2}$ and squaring shows that $A^{1/2}Q_s^{1/2}$ is normal. Its eigenvalues are positive, since it is similar to $Q_s^{1/4}A^{1/2}Q_s^{1/4}$. Hence it is Hermitian, so $Q_s$ commutes with $Q_0=A$. The first part of the proof now applies. For commuting $A,B$ the conclusions follow directly from $Q_t=A^{1-t}B^t$.
\end{proof}

For $0\leq t\leq1$ the bounds in the L\"owner order take the form
\begin{equation}\label{eq:harmonic1}
 \left(\frac{A+D^{-1}}2\right)^{-t}\leq\frac{A^{-1}+Q_t}2,
\end{equation}
\begin{equation}\label{eq:harmonic2}
 \left(\frac{A^{-1}+D}2\right)^{-t}\leq\frac{A+Q_t^{-1}}2.
\end{equation}
Indeed, the harmonic and arithmetic bounds for the geometric mean, together with operator monotonicity of the $t$th power, give
\[
 \left(\frac{A+Q_t^{-1}}2\right)^{-1}\leq A^{-1}\geo Q_t=(A^{-1}\geo D)^t\leq\left(\frac{A^{-1}+D}2\right)^t.
\]
Inversion gives \eqref{eq:harmonic2}. Applying it to $A^{-1},B^{-1}$ gives \eqref{eq:harmonic1}.

\subsection{Power and endpoint comparisons}
Araki's inequality and Weyl's majorant theorem compare the powers of $Q_t$ with those of $N$.
\begin{proposition}\label{prop:powercomp}
Let $A,B\in\Pn$ and $t\in\mathbb R$. Then
\[
 \begin{gathered}
 A^{p/2}N^{2tp}A^{p/2}\logmaj Q_t^p\qquad(0\leq p\leq1),\\
 N^{2t}\logmaj A^{-1/2}Q_tA^{-1/2}.
 \end{gathered}
\]
\end{proposition}
\begin{proof}
Araki's inequality \cite{Araki} applied to $N^{2t}$ and $A$ gives $N^{tp}A^pN^{tp}\logmaj(N^tAN^t)^p$. The left-hand side is isospectral with $A^{p/2}N^{2tp}A^{p/2}$. For the second inequality take $Z=A^{-1/2}N^tA^{1/2}$. Its eigenvalues are those of $N^t$, and $ZZ^*=A^{-1/2}Q_tA^{-1/2}$. Weyl's majorant theorem gives the result on squaring the partial products.
\end{proof}

The identities in \eqref{eq:sub} also compare $Q_t$ with the near geometric curve $C_t=A^{-1}\star_t B$.
\begin{proposition}\label{prop:endpointcomp}
Let $A,B\in\Pn$. If $A\geq I$, then
\[
 C_t\wlogmaj Q_t\quad(0\leq t\leq1),\qquad Q_t\wlogmaj C_t\quad(1\leq t\leq3/2).
\]
If $A\leq I$, then $Q_t\wlogmaj C_t$ for $-1/2\leq t\leq0$.
\end{proposition}
\begin{proof}
For $A\geq I$ and $0\leq t\leq1$, monotonicity in the first variable gives $A^{-1}\geo_t D\leq A\geo_t D$. The geometric--spectral comparison \cite{GT} $A\geo_t D\logmaj A\natural_t D$ proves the first assertion.

For $t\in[-1/2,0]\cup[1,3/2]$, the inequalities of Furuichi and Seo \cite[Proposition~4.1 and Theorem~4.3]{FS} give $A\natural_t D\logmaj A\geo_t D$. Suppose first that $A\geq I$ and $t=1+s$, $0\leq s\leq1/2$. Parameter reversal gives
\[
 A^{\pm1}\geo_{1+s}D=D^{1/2}(D^{1/2}A^{\mp1}D^{1/2})^sD^{1/2}.
\]
Since $A^{-1}\leq A$, the L\"owner--Heinz inequality gives $A\geo_{1+s}D\leq A^{-1}\geo_{1+s}D$.

Finally, suppose $A\leq I$ and $t=-s$, $0\leq s\leq1/2$. A positive contraction $S$ satisfies $\lambda_j(SXS)\leq\lambda_j(X)$ for $X>0$, since $X^{1/2}S^2X^{1/2}\leq X$ and $X^{1/2}S^2X^{1/2}$ is isospectral with $SXS$. For $S\geq I$ the eigenvalue inequalities reverse. Apply these observations first to the inner congruence and then to the outer one in
\[
 A^{\pm1}\geo_{-s}D=A^{\pm1/2}(A^{\pm1/2}D^{-1}A^{\pm1/2})^sA^{\pm1/2}.
\]
The power $s$ preserves the ordering of positive eigenvalues. Hence $\lambda_j(A\geo_{-s}D)\leq\lambda_j(D^{-s})\leq\lambda_j(A^{-1}\geo_{-s}D)$ for each $j$. This proves the assertion by \eqref{eq:sub}.
\end{proof}

Without a normalization on $A$, write $a=\lambda_{\min}(A)$ and $b=\lambda_{\max}(A)$. Applying Proposition~\ref{prop:endpointcomp} to $A/a$ or $A/b$ gives
\begin{align}
 Q_t&\wlogmaj b^{2-2t}C_t&&(-1/2\leq t\leq0),\\
 a^{2-2t}C_t&\wlogmaj Q_t&&(0\leq t\leq1),\\
 Q_t&\wlogmaj a^{2-2t}C_t&&(1\leq t\leq3/2).
\end{align}
Indeed, replacing $A$ by $A/d$ multiplies $C_t$ by $d^{1-t}$ and $Q_t$ by $d^{t-1}$. All three bounds are attained when $A$ is a scalar matrix.

\section{Weighted convex order}\label{sec:weighted}
\subsection{The two weighted spectral measures}
The traces of the curves depend on the eigenvalues of $M$ and $N$ and on the trace of $A$ on their eigenspaces. Write
\[
 M=\sum_i m_iE_i,\qquad N=\sum_j n_jF_j
\]
with $m_1\leq\cdots\leq m_p$. Here the $E_i$ are mutually orthogonal projections summing to $I$, and likewise for the $F_j$. Each eigenspace may be split into orthogonal subspaces, so the same eigenvalue may occur more than once. The measures
\[
 \mu_M=\sum_i\Tr(AE_i)\delta_{m_i},\qquad
 \mu_N=\sum_j\Tr(AF_j)\delta_{n_j}
\]
have the same mass $\Tr A$ and the same first moment by Lemma~\ref{lem:transport}. Inequality \eqref{eq:introweight} says that, after division by $\Tr A$, these probability measures are ordered in the usual convex order \cite[Section~3.A]{SS}.

For the remainder formula, set
\[
 \pi_{ij}=\fnorm{E_iA^{1/2}F_j}^2,\qquad w_{ik}=\fnorm{E_iTE_k}^2=w_{ki},
\]
and denote the Bregman remainder at $m$ with supporting slope $g\in\partial f(m)$ by $D_f(y,m;g)=f(y)-f(m)-g(y-m)$. For orthonormal eigenbases $(u_i)$ of $M$ and $(v_j)$ of $N$, the rank-one projections give
\[
 \pi_{ij}=|\langle u_i,A^{1/2}v_j\rangle|^2,\qquad
 w_{ik}=|\langle u_i,Tu_k\rangle|^2.
\]

\begin{theorem}[Weighted convex order]\label{thm:weighted}
Let $f$ be a finite real convex function on an open interval containing $\spec M\cup\spec N$. Then
\[
 \Tr(Af(M))\leq\Tr(Af(N)).
\]
If $f$ is strictly convex on the convex hull of the two spectra, equality holds if and only if $AB=BA$.
For every choice of the spectral projections above and of slopes $g_i\in\partial f(m_i)$, chosen independently even for equal $m_i$, one has
\begin{equation}\label{eq:bregman}
\begin{split}
 \Tr A[f(N)-f(M)]&=\sum_{i,j}\pi_{ij}D_f(n_j,m_i;g_i)\\
 &\quad+\sum_{i<k}(g_k-g_i)(m_i^{-1}-m_k^{-1})w_{ik}.
\end{split}
\end{equation}
Every summand on the right is nonnegative.
\end{theorem}
\begin{proof}
The row and column sums of $(\pi_{ij})$ are $\Tr(AE_i)$ and $\Tr(AF_j)$, respectively. By the identities $A^{1/2}NA^{1/2}=TM^{-1}T$ and $T^2=MAM$,
\[
 \sum_j n_j\pi_{ij}=\Tr(E_iTM^{-1}T)=\sum_k m_k^{-1}w_{ik},\qquad
 \sum_k w_{ik}=\Tr(E_iT^2)=m_i^2\Tr(AE_i).
\]
In particular,
\[
 \sum_j(n_j-m_i)\pi_{ij}=\sum_k(m_k^{-1}-m_i^{-1})w_{ik}.
\]
The trace difference is $\sum_{i,j}\pi_{ij}[f(n_j)-f(m_i)]$. Subtract the supporting line at $m_i$ in each row. By the last identity, the linear terms sum to $\sum_{i,k}g_i(m_k^{-1}-m_i^{-1})w_{ik}$. Pairing the terms $(i,k)$ and $(k,i)$ gives \eqref{eq:bregman}. Convexity gives $D_f\geq0$ and $g_i\leq g_k$ when $m_i<m_k$. If $m_i=m_k$, the factor $m_i^{-1}-m_k^{-1}$ is zero.

If $f$ is strictly convex and $m_i<m_k$, the secant slope between these points lies strictly between $g_i$ and $g_k$. Equality in the trace comparison therefore forces $E_iTE_k=0$ whenever $m_i\ne m_k$. Thus $T$ preserves every eigenspace of $M$ and commutes with $M$. Since $B=T^2$ and $A=M^{-1}BM^{-1}$, we obtain $AB=BA$. Conversely, when $A$ and $B$ commute, $M=N=A^{-1/2}B^{1/2}$.
\end{proof}
The weight $A$ is needed even for affine test functions. For example, take $M=\operatorname{diag}(1,2)$, $T=\left(\begin{smallmatrix}2&2\\2&4\end{smallmatrix}\right)$, $B=T^2$, and $A=M^{-1}BM^{-1}$. Then $M=A^{-1}\geo B$. From \eqref{eq:transport}, $\Tr N=\Tr(A^{-1}TM^{-1}T)=15/2$, whereas $\Tr M=3$. The affine test $f(x)=-x$ violates the unweighted inequality.

Theorem~\ref{thm:weighted} also shows that the spectrum of $M$ lies in $[\ell,L]=[\lambda_{\min}(N),\lambda_{\max}(N)]$. Indeed, the distance to $[\ell,L]$ is a nonnegative convex function which vanishes on $\spec N$, and every weight $\Tr(AE_i)$ is positive.

\subsection{The integral remainder and equality}
For $c\geq0$, define
\[
 H_{A,B}(c)=\int(x-c)_+\,d(\mu_N-\mu_M)(x)
 =\Tr A[(N-cI)_+-(M-cI)_+].
\]
This function is continuous and piecewise linear. It vanishes for $c\geq L$ by the spectral bounds and for $c\leq\ell$ by equality of the masses and first moments. Theorem~\ref{thm:weighted} gives a nonnegative expression for it. With
\[
 K_{m,n}(c)=(n-c)_+-(m-c)_+-\ind{m>c}(n-m),
\]
one has
\begin{equation}\label{eq:kernel}
 H_{A,B}(c)=\sum_{i,j}\pi_{ij}K_{m_i,n_j}(c)+
 \sum_{i<k}(m_i^{-1}-m_k^{-1})w_{ik}\ind{m_i\leq c<m_k}\geq0.
\end{equation}
This is \eqref{eq:bregman} with $f(x)=(x-c)_+$ and $g_i=\ind{m_i>c}$.
It is convenient to introduce the finite positive measure on $(0,\infty)$
\[
 d\eta_{A,B}(c)=H_{A,B}(c)\,dc,
\]
which is supported in $[\ell,L]$. We refer to it as the positive remainder measure for the weighted convex order.

\begin{proposition}\label{prop:measure}
For every $f\in C^2([\ell,L])$,
\begin{equation}\label{eq:integral}
 \Tr A[f(N)-f(M)]=\int_\ell^L f''(c)H_{A,B}(c)\,dc=\int_{[\ell,L]} f''(c)\,d\eta_{A,B}(c).
\end{equation}
For a finite convex function $f$ on a neighborhood of $[\ell,L]$, the same formula holds with $f''(c)\,dc$ replaced by the distributional second derivative $d\nu_f(c)$ on $(\ell,L)$. In this case the integral is $\int_{(\ell,L)} H_{A,B}(c)\,d\nu_f(c)$. The mass of the positive measure $\eta_{A,B}$ is
\begin{equation}\label{eq:mass}
 \int_\ell^L H_{A,B}(c)\,dc=\tfrac12\Tr A(N^2-M^2)
 =\tfrac14\fnorm{A^{-1/2}[A,B^{1/2}]A^{-1/2}}^2.
\end{equation}
\end{proposition}
\begin{proof}
Taylor's formula writes $f(x)$ as the sum of its tangent at $\ell$ and $\int_\ell^L(x-c)_+f''(c)\,dc$. Integrating against $\mu_N-\mu_M$ cancels the affine part and gives \eqref{eq:integral}. For a convex $f$, integration of its right derivative gives
\[
 f(x)=f(\ell)+f'_+(\ell)(x-\ell)+\int_{(\ell,L)}(x-c)_+\,d\nu_f(c).
\]
The same cancellation proves the second assertion. Taking $f(x)=x^2/2$ gives the first equality in \eqref{eq:mass}. For the second, cyclicity gives $\Tr(A[N^2-M^2])=\Tr(NAN-MAM)=\Tr(D-B)$. Now apply \eqref{eq:endpoint}. If $\ell=L$, both spectral measures agree, and all the integrals are zero.
\end{proof}

\begin{remark*}
The representation \eqref{eq:integral} is related to the positive second-order representations in \cite{DLVBMV}, which developed the pinching correction proposed in \cite{DinhBMV}. In that setting, writing $D_0=E_A(B)$ for the spectral pinching of $B$ relative to a Hermitian matrix $A$, the trace difference for the pencil $xA+yB$ is represented by integrating $f''(xa+yb)$ against a positive planar measure of mass $\tfrac12\fnorm{B-D_0}^2$, with a prefactor $y^2$. In the present setting, the positive measure arises instead from the weighted convex order of $M$ and $N$, and its mass is the commutator expression \eqref{eq:mass}. In particular, taking $f(x)=x^{k+2}/((k+1)(k+2))$ in \eqref{eq:integral} gives the moment identities
\begin{equation*}
 \int_{(0,\infty)}c^k\,d\eta_{A,B}(c)
 =\frac{\Tr A(N^{k+2}-M^{k+2})}{(k+1)(k+2)},\qquad k=0,1,2,\ldots.
\end{equation*}
\end{remark*}

Set
\begin{equation}\label{eq:CA}
 C_{A,B}=\tfrac14\fnorm{A^{-1/2}[A,B^{1/2}]A^{-1/2}}^2.
\end{equation}
For any $f\in C^2([\ell,L])$ with $\kappa\leq f''\leq K$, \eqref{eq:integral} and \eqref{eq:mass} give
\begin{equation}\label{eq:bounds}
 \kappa C_{A,B}\leq\Tr A[f(N)-f(M)]\leq K C_{A,B}.
\end{equation}
In particular, if $s>0$, the resolvent difference $R(s)=\Tr A[(N+sI)^{-1}-(M+sI)^{-1}]$ satisfies
\begin{equation}\label{eq:resolvent}
 \frac{2C_{A,B}}{(s+L)^3}\leq R(s)=2\int_\ell^L\frac{H_{A,B}(c)}{(s+c)^3}\,dc\leq\frac{2C_{A,B}}{(s+\ell)^3}.
\end{equation}
The corresponding finite sum in \eqref{eq:bregman} is
\begin{equation}\label{eq:resolventsum}
\begin{split}
 R(s)&=\sum_{i,j}\pi_{ij}\frac{(n_j-m_i)^2}{(s+m_i)^2(s+n_j)}\\
 &\quad+\sum_{i<k}(m_i^{-1}-m_k^{-1})\bigl((s+m_i)^{-2}-(s+m_k)^{-2}\bigr)w_{ik}.
\end{split}
\end{equation}
Both formulas follow by taking $f(x)=(s+x)^{-1}$. The lower bound in \eqref{eq:resolvent} gives $R(s)>0$ for every $s>0$ when $AB\ne BA$. Taking $f(x)=e^{ux}$ instead gives $\Tr(Ae^{uM})\leq\Tr(Ae^{uN})$ for every real $u$, strictly when $u\ne0$ and $AB\ne BA$.

For a convex $f$ with affine intervals, equality depends on the common reducing subspaces of $A$ and $B$. We call the pair irreducible if its only common reducing subspaces are $\{0\}$ and the whole space.
\begin{proposition}\label{prop:equality}
Let $A=\bigoplus_j A_j$, $B=\bigoplus_j B_j$ be a common orthogonal decomposition into irreducible pairs. Write $I_j$ for the closed interval from the least to the greatest eigenvalue of $N_j=A_j^{-1}\star B_j$. For a finite convex function $f$ as in Theorem~\ref{thm:weighted}, equality in \eqref{eq:introweight} holds if and only if $f$ is affine on each $I_j$.
\end{proposition}
\begin{proof}
First suppose that the pair is irreducible. We show that $H_{A,B}(c)>0$ whenever $\ell<c<L$. If $c<\lambda_{\min}(M)$, equality of the first moments gives $H_{A,B}(c)=\Tr(A(cI-N)_+)>0$. If $c\geq\lambda_{\max}(M)$, then $H_{A,B}(c)=\Tr(A(N-cI)_+)>0$. In the remaining case the spectral projection $E$ of $M$ for $(-\infty,c]$ is nonzero and proper. If $H_{A,B}(c)=0$, the second sum in \eqref{eq:kernel} gives $ET(I-E)=0$. Since $T$ is Hermitian, $E$ commutes with $T$ and with $M$, and hence with $B=T^2$ and $A=M^{-1}BM^{-1}$. This contradicts irreducibility. Thus $\int H_{A,B}\,d\nu_f=0$ exactly when $\nu_f$ vanishes on $(\ell,L)$, or equivalently when $f$ is affine on $[\ell,L]$. The assertion is automatic if $\ell=L$. For a general pair the trace difference is the sum of the nonnegative differences on the irreducible summands.
\end{proof}

\subsection{The trace gap}
For $f(x)=x^{2t}$, the sign of $f''$ is constant on $(0,\infty)$. Formula \eqref{eq:integral} therefore gives both the sign and the size of $\Tr(Q_t-P_t)$.
\begin{theorem}\label{thm:trace}
Set $P_t=M^tAM^t$, $Q_t=N^tAN^t$, and $F(t)=\Tr(Q_t-P_t)$. For all $t\in\mathbb R$,
\begin{equation}\label{eq:trace}
 F(t)=2t(2t-1)\int_\ell^L c^{2t-2}H_{A,B}(c)\,dc.
\end{equation}
If $AB\ne BA$, then $F(t)<0$ for $0<t<1/2$ and $F(t)>0$ for $t<0$ or $t>1/2$. In this case its two zeros are simple. Set $a_t=\min_{\ell\leq c\leq L}c^{2t-2}$ and $b_t=\max_{\ell\leq c\leq L}c^{2t-2}$. Then
\begin{equation}\label{eq:tracebounds}
 |2t(2t-1)|a_t C_{A,B}\leq |F(t)|\leq |2t(2t-1)|b_t C_{A,B}.
\end{equation}
Consequently, for any $t\notin\{0,1/2\}$, $P_t$ and $Q_t$ are isospectral if and only if $AB=BA$.
\end{theorem}
\begin{proof}
By cyclicity, $F(t)=\Tr(A[N^{2t}-M^{2t}])$. Substitution of $f(x)=x^{2t}$ in \eqref{eq:integral} gives \eqref{eq:trace}. The integral there lies between $a_tC_{A,B}$ and $b_tC_{A,B}$ by \eqref{eq:mass}. Since $\ell>0$, it is strictly positive whenever $AB\ne BA$. The integrand and its $t$-derivative are continuous on the compact spectral interval. Differentiation at the two zeros gives
\[
 F'(0)=-2\int_\ell^L c^{-2}H_{A,B}(c)\,dc,\qquad
 F'(1/2)=2\int_\ell^L c^{-1}H_{A,B}(c)\,dc.
\]
These derivatives have opposite signs if $AB\ne BA$. Finally, equal spectra imply equal traces, while $AB=BA$ gives $M=N$ and hence $P_t=Q_t$ for all $t$.
\end{proof}

\section{Log-majorization of the two curves}\label{sec:log}
For $X\in\Pn$, the eigenvalues of $\wh X$ are the products of $k$ distinct entries of the eigenvalue list of $X$. In particular,
\begin{equation}\label{eq:wedge}
 \lambda_{\max}(\wh X)=\prod_{j=1}^k\lambda_j(X).
\end{equation}
Exterior powers preserve products, adjoints and real powers of positive definite matrices. They also preserve the L\"owner order. Indeed, if $0<X\leq Y$, all eigenvalues of $Y^{-1/2}XY^{-1/2}$ lie in $(0,1]$. The same holds for its exterior powers. Congruence by $\wh(Y^{1/2})$ gives $\wh X\leq\wh Y$. To recover the largest eigenvalue from the trace comparison we pass to symmetric powers and take a limit.

\begin{theorem}[Log-majorization reversal]\label{thm:logmaj}
Let $A,B\in\Pn$, $P_t=M^tAM^t$, and $Q_t=N^tAN^t$. Then
\begin{equation}\label{eq:logmaj}
\begin{aligned}
 Q_t&\logmaj P_t &&(0<t<1/2),\\
 P_t&\logmaj Q_t &&(t<0\ \text{or}\ t>1/2).
\end{aligned}
\end{equation}
At $t=0$ the two matrices coincide, and at $t=1/2$ they have the same eigenvalues. At every other parameter their spectra agree if and only if $A$ and $B$ commute.
\end{theorem}
\begin{proof}
Fix $1\leq k\leq n$. For $q=1,2,\ldots$, let $\Phi(C)=\Sym^q(\wh C)$. Both tensor representations preserve products and adjoints. Diagonalization of $C>0$ gives $\Phi(C^s)=\Phi(C)^s$ for $s\in\mathbb R$, since the eigenvalues of $\Phi(C)$ are monomials in those of $C$. Thus the two transports of $\Phi(A),\Phi(B)$ are $\Phi(M),\Phi(N)$.

Suppose $0<t<1/2$. Apply Theorem~\ref{thm:weighted} to the amplified pair with the convex function $f(x)=-x^{2t}$. Cyclicity of the trace gives
\[
 \Tr\Sym^q(\wh Q_t)=\Tr\Phi(A)\Phi(N)^{2t}\leq\Tr\Phi(A)\Phi(M)^{2t}=\Tr\Sym^q(\wh P_t).
\]
For $t<0$ or $t>1/2$, use $f(x)=x^{2t}$ to reverse the inequality. If $E>0$ acts on a $d$-dimensional space, $\Sym^q(E)$ has largest eigenvalue $\lambda_{\max}(E)^q$ and dimension $\binom{q+d-1}{d-1}$. Hence
\[
 \lambda_{\max}(E)^q\leq\Tr\Sym^q(E)\leq\binom{q+d-1}{d-1}\lambda_{\max}(E)^q.
\]
For fixed $k$ and $n$, the binomial coefficient grows only polynomially in $q$. Thus the $q$th root of the trace converges to $\lambda_{\max}(E)$. Applying this limit to $E=\wh Q_t$ and $E=\wh P_t$ gives the required inequality between their largest eigenvalues. These are the products of the $k$ largest eigenvalues of $Q_t$ and $P_t$. For $k=n$ there is equality because $\det M=\det N$.

The equality at $t=0$ follows from $P_0=Q_0=A$, and the midpoint comparison follows from Lemma~\ref{lem:transport}. At every other parameter, equality of the spectra gives equality of the traces, so Theorem~\ref{thm:trace} implies $AB=BA$. Conversely, commuting inputs give $M=N$.
\end{proof}
Equivalently, for $X_r=A^{1/2}M^rA^{1/2}$ and $Y_r=A^{1/2}N^rA^{1/2}$, one has $Y_r\logmaj X_r$ for $0<r<1$ and $X_r\logmaj Y_r$ for $r<0$ or $r>1$. Indeed, $M^tA^{1/2}$ and its adjoint give the two products $P_t$ and $X_{2t}$, and the same argument applies to $Q_t$ and $Y_{2t}$.

Matrix equality is more restrictive than equality of the midpoint spectra. For any $t\ne0$, $P_t=Q_t$ is equivalent to $AB=BA$. Indeed, if $P_t=Q_t$ at any $t\ne0$, then $M^t$ and $N^t$ are positive solutions of $YAY=P_t$. Uniqueness gives $M^t=N^t$, hence $M=N$ and $P_1=Q_1$. Theorem~\ref{thm:trace} at $t=1$ now gives $AB=BA$. In particular, for noncommuting inputs and $t\ne0$, the two curve matrices are incomparable in both the L\"owner and near orders. Indeed, $X\near Y$ means $X^{-1}\geo Y\geq I$ (see \cite{DFnear}); if also $\det X=\det Y$, then $\det(X^{-1}\geo Y)=1$. All its eigenvalues must be one, so $X=Y$.

\begin{corollary}\label{cor:strict}
Let $AB\ne BA$ and $t\notin\{0,1/2\}$. At least one partial product inequality in \eqref{eq:logmaj} is strict. In dimension two the largest-eigenvalue inequality is strict.
\end{corollary}
\begin{proof}
Equality of all ordered partial products would give equality of the spectra. In dimension two only the first partial product can differ.
\end{proof}
For $n>2$, any one of the $n-1$ inequalities can be the only strict one. Fix $t\notin\{0,1/2\}$, an index $k<n$, and a noncommuting $2\times2$ pair. Adjoin to both inputs $k-1$ copies of a scalar larger than all four eigenvalues of the two resulting $2\times2$ curve matrices at time $t$, and $n-k-1$ copies of a smaller positive scalar below all four eigenvalues. The curves respect direct sums and are constant on the added blocks. The two eigenvalues from the original pair occupy positions $k,k+1$. Their product is the same on both curves, while their larger members differ by Corollary~\ref{cor:strict}.

\subsection{Geodesic convex hulls of unitary orbits}
For $C\in\Pn$ write
\[
 \mathcal M(C)=\{X\in\Pn:X\logmaj C\},\qquad
 \rho(C)=\fnorm{\log C-\frac{\log\det C}{n}I},
\]
where the subtracted scalar is the mean of the log-eigenvalues. For the affine-invariant Riemannian metric, $\mathcal M(C)$ is the geodesic convex hull of the unitary orbit of $C$ \cite[Theorem~2.6]{DAT}. For this metric $\rho(C)$ is the distance from $(\det C)^{-1/n}C$ to $I$.
\begin{corollary}\label{cor:geometry}
For $0<t<1/2$ one has $\mathcal M(Q_t)\subseteq\mathcal M(P_t)$ and $\rho(Q_t)\leq\rho(P_t)$. For $t<0$ or $t>1/2$ both comparisons reverse. At $t=0$ and $t=1/2$ both are equalities. At all other parameters the inclusion is proper and the distance inequality is strict whenever $AB\ne BA$.
\end{corollary}
\begin{proof}
Let $X\logmaj Y$. Transitivity gives $\mathcal M(X)\subseteq\mathcal M(Y)$. If the spectra differ, then $Y\notin\mathcal M(X)$, so the inclusion is proper. The ordered log-eigenvalue vectors of $X$ and $Y$ have the same mean. Subtracting it preserves their majorization relation, and applying the strictly convex function $s\mapsto s^2$ gives $\rho(X)^2\leq\rho(Y)^2$, strictly for unequal spectra. Theorem~\ref{thm:logmaj} now gives both assertions.
\end{proof}
The unitary-orbit characterization of majorization \cite{BhatiaMA} also gives, for $0<t<1/2$,
\[
 \log Q_t\in\operatorname{conv}\{U(\log P_t)U^*:U\in U(n)\}.
\]
Equivalently, the vector of log-eigenvalues of $Q_t$ belongs to the convex hull of the permutations of the log-eigenvalue vector of $P_t$. For $t<0$ or $t>1/2$, interchange $P_t$ and $Q_t$.

\subsection{An elementary proof in dimension two}
In dimension two the two transports have the same determinant. Their traces therefore determine how their spectral intervals are nested.
\begin{proposition}\label{prop:two}
Let $A,B\in\mathbb P_2$, and put $h=\Tr(AM)=\Tr(AN)$. Then
\begin{equation}\label{eq:two}
 h(\Tr N-\Tr M)=\Tr D-\Tr B.
\end{equation}
If $AB\ne BA$, both eigenvalues of $M$ lie strictly between the eigenvalues of $N$. Consequently, for every convex function $f$ on the spectral interval of $N$,
\[
 \Tr(Af(M))\leq\Tr(Af(N)),
\]
with strict inequality when $AB\ne BA$ and $f$ is strictly convex.
\end{proposition}
\begin{proof}
The determinants of $M$ and $N$ agree. Apply Cayley--Hamilton to the two transports, multiply by $A$, and take traces. The difference is \eqref{eq:two}, since $MAM=B$ and $NAN=D$. If $AB\ne BA$, the endpoint identity \eqref{eq:endpoint} gives $\Tr N>\Tr M$. For two positive numbers with fixed product, increasing their sum increases the larger and decreases the smaller.

If $A$ and $B$ commute, then $M=N$. Otherwise let $L_0$ be the affine function which agrees with $f$ at the two eigenvalues of $N$. Convexity gives $f(M)\leq L_0(M)$, while $L_0(N)=f(N)$. The common first moment gives $\Tr(AL_0(M))=\Tr(AL_0(N))$. For strictly convex $f$ and noncommuting inputs, $L_0(M)-f(M)>0$.
\end{proof}
Taking $f(x)=x^{2t}$ or $-x^{2t}$ gives the trace comparison for $P_t$ and $Q_t$ on each of the intervals cut out by $0$ and $1/2$. Since $P_t$ and $Q_t$ have the same determinant, their larger eigenvalues are ordered in the same way as their traces. This proves the log-majorization reversal in dimension two without tensor powers.

The following formulas will also be used in the examples. For $A,B\in\mathbb P_2$, put $\delta=\sqrt{\det A\det B}$, $c=(\det B/\det A)^{1/4}$ and $h=\sqrt{\Tr(AB)+2\delta}$. By the Cayley--Hamilton formula for the square root,
\[
 N=\frac{A^{-1/2}(TAT+\delta I)A^{-1/2}}{h},\qquad \Tr(AN)=h.
\]
For $S=M,N$, put $u_S=\arcosh(\Tr S/(2c))$ and $f_a(u)=\sinh(au)/\sinh u$, with $f_a(0)=a$. Interpolation at the eigenvalues $e^{\pm u_S}$ of $S/c$ gives $(S/c)^a=f_a(u_S)S/c+f_{1-a}(u_S)I$ for all real $a$. In particular, with $F_t=f_t(u_N)N/c+f_{1-t}(u_N)I$,
\begin{equation}\label{eq:twoQ}
 Q_t=c^{2t}F_tAF_t,
\end{equation}
and, since $\Tr(AM)=\Tr(AN)=h$,
\begin{equation}\label{eq:twoTrace}
 \Tr(S^tAS^t)=c^{2t}\left[\frac{h}{c}f_{2t}(u_S)+(\Tr A)f_{1-2t}(u_S)\right],\qquad S=M,N.
\end{equation}

\begin{example}\label{ex:endpoint}
Take
\[
 A=\begin{pmatrix}4&0\\0&1\end{pmatrix},\qquad B=\begin{pmatrix}5&4\\4&5\end{pmatrix},\qquad B^{1/2}=\begin{pmatrix}2&1\\1&2\end{pmatrix}.
\]
Then $D=\begin{pmatrix}17/4&5\\5&8\end{pmatrix}$. Both $D$ and $B$ have determinant $9$, and
\[
 A^{-1/2}[A,B^{1/2}]A^{-1/2}=\begin{pmatrix}0&3/2\\-3/2&0\end{pmatrix},
\]
so $\Tr D-\Tr B=9/4$ and $C_{A,B}=9/8$. The eigenvalues are
\[
 \lambda(B)=(9,1),\qquad \lambda(D)=\left(\frac{49+5\sqrt{73}}8,\frac{49-5\sqrt{73}}8\right).
\]
In particular, $B\logmaj D$ is strict at the first eigenvalue.
\end{example}

\section{The log-Euclidean mean and powers of the inputs}\label{sec:power}
\subsection{The log-Euclidean lower bound}
Put $L_t(A,B)=\exp((1-t)\log A+t\log B)$. For the near geometric curve one has
\begin{equation}\label{eq:LEnear}
 L_t(A,B)\logmaj (A^{q(1-t)/2}B^{qt}A^{q(1-t)/2})^{1/q}
 \logmaj(A^q\star_t B^q)^{1/q}\qquad(0\leq t\leq1,\ q>0).
\end{equation}
Indeed, Araki's inequality \cite{Araki} shows that the middle term is increasing in $q$ with respect to log-majorization. Its limit as $q\downarrow0$ is $L_t(A,B)$ by the Lie product formula. The second comparison follows by applying \cite[Theorem~5.2]{FGK} to $A^q,B^q$ and taking the power $1/q$. For $Q_t$, the lower bound holds on the two ranges $t\geq1/2$ and $t\leq-1/2$. On $0<t<1/2$ there are two-dimensional pairs giving strict comparison in either direction.
\begin{lemma}\label{lem:normtest}
Let $F$ and $G$ map pairs of positive definite matrices to positive definite matrices of the same size in every dimension. Suppose that both maps are homogeneous of degree one under common scaling and commute with all exterior powers. If in every dimension
\begin{equation}\label{eq:normtest}
 G(A,B)\leq I\quad\Longrightarrow\quad F(A,B)\leq I,
\end{equation}
then $F(A,B)\wlogmaj G(A,B)$. If their determinants also agree, then $F(A,B)\logmaj G(A,B)$.
\end{lemma}
\begin{proof}
Set $c=\norm{G(A,B)}$. Applying \eqref{eq:normtest} to $(A/c,B/c)$ yields $\norm{F(A,B)}\leq c$. Apply the same norm bound to $(\wh A,\wh B)$ and use \eqref{eq:wedge} to obtain the $k$th partial-product inequality. Equality at $k=n$ is precisely equality of the determinants.
\end{proof}
\begin{theorem}\label{thm:LE}
Let $A,B\in\Pn$ and $q>0$. If $t\geq1/2$ or $t\leq-1/2$, then
\[
 L_t(A,B)\logmaj(A^q\ns_t B^q)^{1/q}.
\]
\end{theorem}
\begin{proof}
It suffices to take $q=1$, since $L_t(A^q,B^q)=L_t(A,B)^q$. Write $H=(B^{1/2}AB^{1/2})^{1/2}$, so that $N=A^{-1/2}HA^{-1/2}$. We first show that $Q_t\leq I$ implies $L_t(A,B)\leq I$.

Suppose $Q_t\leq I$. Then $A\leq N^{-2t}$. For $t>1/2$, the L\"owner--Heinz inequality and inversion give $N\leq A^{-1/(2t)}$. For $t\leq-1/2$, the L\"owner--Heinz inequality gives $N\geq A^{-1/(2t)}$. Thus, with $s=1-1/(2t)>0$, we have $H\leq A^s$ in the first case and $H\geq A^s$ in the second. The two forms of Furuta's inequality \cite{Furuta}, with $p=2$, $r=1/s$ and outer exponent $1/2$, now give
\[
 (A^{1/2}H^2A^{1/2})^{1/2}
 \begin{cases}
 \leq A^{s+1/2},&t>1/2,\\
 \geq A^{s+1/2},&t\leq-1/2.
 \end{cases}
\]
Indeed, $(A^s)^{r/2}=A^{1/2}$, and the exponent $1/2$ satisfies the Furuta condition $1/2\leq(1+r)/(2+r)$. The left-hand side equals $A^{1/2}B^{1/2}A^{1/2}$. Congruence by $A^{-1/2}$ now gives
\begin{equation}\label{eq:LEpower}
 B^{1/2}\leq A^{(t-1)/(2t)}\quad(t>1/2),\qquad
 B^{1/2}\geq A^{(t-1)/(2t)}\quad(t\leq-1/2).
\end{equation}
Taking logarithms and multiplying by $2t$ gives $t\log B\leq(t-1)\log A$ in both cases. Hence $L_t(A,B)\leq I$.

Both curves are homogeneous under common scaling. Exterior powers commute with the products and real powers defining $Q_t$. For $L_t$, write $\rho=\wh$ and let $d\rho$ be its differential at $I$. Then $\rho(e^H)=e^{d\rho(H)}$ and $d\rho$ is linear, so $L_t$ also commutes with exterior powers. Both determinants are $(\det A)^{1-t}(\det B)^t$. Lemma~\ref{lem:normtest} gives $L_t(A,B)\logmaj Q_t$, and continuity in $t$ gives the case $t=1/2$.
\end{proof}
\begin{proposition}\label{prop:LEcounter}
Fix $0<t<1/2$. There are pairs in $\mathbb P_2$ for which $L_t(A,B)\not\prec_{\log}Q_t$, and pairs for which $Q_t\not\prec_{\log}L_t(A,B)$. The two counterexamples can be chosen with the same $A$ and unitarily similar second inputs.
\end{proposition}
\begin{proof}
Put $b=(1-t)/t$. For $u>0$ and $-1<c<1$, take
\[
 Z_c=\begin{pmatrix}c&\sqrt{1-c^2}\\\sqrt{1-c^2}&-c\end{pmatrix},\quad A=e^{uZ_1},\quad B=e^{buZ_c}.
\]
Both inputs have determinant one. Since $(Z_1+Z_c)^2=(2+2c)I$,
\begin{equation}\label{eq:LElimit}
 u^{-1}\log\lambda_{\max}(L_t(A,B))=(1-t)\sqrt{2+2c}.
\end{equation}
We claim that $u^{-1}\log\lambda_{\max}(Q_t)\to 2t$ as $u\to\infty$, for every fixed $c\in(-1,1)$.

Write $C=B^{1/2}$, $H=(CAC)^{1/2}$ and $N=A^{-1/2}HA^{-1/2}$. For positive functions of $u$, write $f\asymp g$ if their ratio stays between two positive constants as $u\to\infty$. The square-root formula for a positive $2\times2$ matrix of determinant one gives
\[
 H=\frac{CAC+I}{\sqrt{\Tr(AB)+2}},\qquad
 \Tr(AN)=\Tr H=\sqrt{\Tr(AB)+2}\asymp e^{(1+b)u/2}.
\]
Indeed, $B=\cosh(bu)I+\sinh(bu)Z_c$ and $c>-1$. Moreover,
\[
 \Tr(A^{-1}CAC)=C_{11}^2+C_{22}^2+(e^{2u}+e^{-2u})C_{12}^2\asymp e^{(2+b)u},
\]
since $C_{12}=\sqrt{1-c^2}\sinh(bu/2)$. Thus $\Tr N\asymp e^{hu}$, where $h=(3+b)/2$. If $e^{\pm v}$ are the eigenvalues of $N$, with $v\geq0$, then $v=hu+O(1)$. Interpolating $x^{2t}$ at $e^{\pm v}$ gives
\[
 \Tr Q_t=\frac{\sinh((1-2t)v)}{\sinh v}\Tr A+\frac{\sinh(2tv)}{\sinh v}\Tr(AN).
\]
The two terms are positive and have orders $e^{(1-2th)u}$ and $e^{(2th-1)u}$, respectively. Since $2th=1+2t$, it follows that $\Tr Q_t\asymp e^{2tu}$. The bounds $\tfrac12\Tr Q_t\leq\lambda_{\max}(Q_t)\leq\Tr Q_t$ prove the claim.

As $c$ varies in $(-1,1)$, the right-hand side of \eqref{eq:LElimit} ranges over $(0,2(1-t))$, which contains $2t$. Choose $c_1,c_2$ so that the two values in \eqref{eq:LElimit} lie on opposite sides of $2t$, and take $u$ sufficiently large for both choices. The largest eigenvalues are ordered in opposite directions. Since the determinants agree, this proves both assertions.
\end{proof}
\begin{example}\label{ex:LE}
Even with both inputs fixed, the direction of comparison can change before the midpoint. Take
\[
 A=\begin{pmatrix}2&1\\1&2\end{pmatrix},\qquad B=\begin{pmatrix}5&0\\0&1\end{pmatrix}.
\]
The following intervals enclose the largest eigenvalues:
\[
\begin{array}{c|cc}
 t&\lambda_{\max}(L_t(A,B))&\lambda_{\max}(Q_t)\\\hline
 1/4&(2.9201,2.9202)&(2.8753,2.8754)\\
 49/100&(3.1840,3.1841)&(3.2274,3.2275)
\end{array}
\]
To compute $\lambda_{\max}(L_t)$, put $a=(1-t)\log3$ and $b=t\log5$. The two eigenvalues of $(1-t)\log A+t\log B$ are $(a+b\pm\sqrt{a^2+b^2})/2$. For $Q_t$, write $\tau=\Tr Q_t$. Then
\[
 \lambda_{\max}(Q_t)=\tfrac12(\tau+\sqrt{\tau^2-4\,3^{1-t}5^t}).
\]
Formula \eqref{eq:twoTrace} gives $\tau$ with $c=(5/3)^{1/4}$, $h=\sqrt{12+2\sqrt{15}}$, and $\Tr N=(24-2\sqrt5+4\sqrt{15})/(3h)$. The two determinants agree at each $t$. Thus $Q_t\logmaj L_t(A,B)$ at $t=1/4$, whereas $L_t(A,B)\logmaj Q_t$ at $t=49/100$, both strictly at the first eigenvalue.
\end{example}

Theorem~\ref{thm:LE} and Proposition~\ref{prop:LEcounter} leave open the universal comparison for $-1/2<t<0$; at $t=0$ the two matrices coincide.

\subsection{Powers of the inputs}
Write $Q_t(A,B)=A\ns_t B$ when it is necessary to specify the inputs. For every fixed real $t$,
\begin{equation}\label{eq:LTQ}
 \lim_{p\to0}Q_t(A^p,B^p)^{1/p}=L_t(A,B).
\end{equation}
Indeed, $A^{-p}\star B^p=I+\frac p2(\log B-\log A)+O(p^2)$ by the defining square root. Hence $Q_t(A^p,B^p)=I+p((1-t)\log A+t\log B)+O(p^2)$, which gives \eqref{eq:LTQ} on taking logarithms.

\begin{theorem}[Input powers at the midpoint]\label{thm:power}
Fix $r>1$. Among $t>0$, the inequality
\begin{equation}\label{eq:inputpower}
 Q_t(A,B)^r\logmaj Q_t(A^r,B^r)
\end{equation}
holds for all positive definite matrices $A,B$ of every size if and only if $t=1/2$. For $0<t<1/2$ it can fail in dimension two, and for $t>1/2$ it can fail in dimension three. In both ranges the largest-eigenvalue inequality itself can fail.
\end{theorem}
\begin{proof}
At the midpoint, Lemma~\ref{lem:transport} gives, for every $p>0$,
\[
 Q_{1/2}(A^p,B^p)^{1/p}\eqs(A^{p/2}B^pA^{p/2})^{1/(2p)}.
\]
Araki's log-majorization \cite{Araki,Hiai} therefore gives
\begin{equation}\label{eq:midpower}
 Q_{1/2}(A^p,B^p)^{1/p}\logmaj Q_{1/2}(A^q,B^q)^{1/q}\qquad(0<p\leq q).
\end{equation}
Taking $p=1$ and $q=r$ proves \eqref{eq:inputpower}. Taking $p=s$ and $q=1$ gives the reciprocal-exponent form
\[
 Q_{1/2}(A^s,B^s)\logmaj Q_{1/2}(A,B)^s\qquad(0<s\leq1).
\]
Now fix $0<t<1/2$. Suppose the largest-eigenvalue inequality in \eqref{eq:inputpower} held for all $A,B\in\mathbb P_2$. Apply it successively to $(A^{r^{-j}},B^{r^{-j}})$ for $j=1,\ldots,m$. With $p=r^{-m}$ this gives
\[
 \lambda_{\max}(Q_t(A^p,B^p))^{1/p}\leq\lambda_{\max}(Q_t(A,B)).
\]
Letting $m\to\infty$ in \eqref{eq:LTQ} contradicts Proposition~\ref{prop:LEcounter}, since in dimension two the determinants agree.

Now fix $t>1/2$. We choose a matrix $C$ for which $C_{13}=0$ but $(C^{1/r})_{13}\ne0$. A diagonal scaling then produces different growth rates for the two largest eigenvalues under comparison. For $s\geq1$, let
\[
 S=\operatorname{diag}(s,1,s^{-1}),\quad A=S^2,\quad C=\begin{pmatrix}2&1&0\\1&2&1\\0&1&2\end{pmatrix},\quad T=C^{1/r},\quad B=T^2.
\]
The eigenvalues of $C$ are $2+\sqrt2,2,2-\sqrt2$. Interpolation at these three points, or direct diagonalization, gives
\[
 T_{13}=\tfrac14((2+\sqrt2)^{1/r}+(2-\sqrt2)^{1/r}-2\cdot2^{1/r})<0
\]
by strict concavity. Put $N=S^{-1}(TS^2T)^{1/2}S^{-1}$. Since $TS^2T\geq s^2(Te_1)(Te_1)^*$, operator monotonicity of the square root gives
\[
 \norm N\geq N_{33}\geq cs^3,\qquad c=|T_{31}|^2/\norm{Te_1}>0.
\]
It follows from $A\geq s^{-2}I$ that $\norm{N^tAN^t}^r\geq c^{2tr}s^{r(6t-2)}$.

For the powered inputs the transport is $K=S^{-r}|S^rC|S^{-r}$, where $|Z|=(Z^*Z)^{1/2}$. Cauchy--Schwarz gives $\langle e_j,|S^rC|e_j\rangle\leq\norm{S^rCe_j}$. As $C_{13}=0$, the three diagonal entries of $K$ are respectively $O(s^{-r})$, $O(s^r)$ and $O(s^{2r})$. Hence $\norm K\leq\Tr K=O(s^{2r})$ as $s\to\infty$. Write $R_u=K^uA^rK^u$. The midpoint identity gives $\norm{R_{1/2}}=\norm{S^rC}=O(s^r)$, so
\[
 \norm{R_t}\leq\norm{K}^{2t-1}\norm{R_{1/2}}=O(s^{r(4t-1)}).
\]
Since $r(6t-2)-r(4t-1)=r(2t-1)>0$, $\norm{Q_t(A,B)}^r>\norm{Q_t(A^r,B^r)}$ for all sufficiently large $s$.
\end{proof}

\begin{example}\label{ex:power}
Write $X=Q_t(A,B)^r$, $Y=Q_t(A^r,B^r)$, and $p_k(C)=\prod_{j=1}^k\lambda_j(C)$. For $t=2/5$, $r=2$, take
\[
 A=\begin{pmatrix}2&0\\0&1\end{pmatrix},\qquad B=\begin{pmatrix}2&1\\1&2\end{pmatrix}.
\]
Then $4.305<p_1(X)<4.306$ and $4.283<p_1(Y)<4.284$. Since $X$ and $Y$ have the same determinant, $Y\logmaj X$.

The failure in dimension three can occur at the second partial product even when the largest eigenvalues have the proposed order. Take $t=3/4$, $r=3/2$ and
\[
 A=\frac1{100}\begin{pmatrix}119&-530&543\\-530&4356&-3892\\543&-3892&3574\end{pmatrix},\quad
 B=\frac1{100}\begin{pmatrix}3293&3306&-1969\\3306&3337&-1983\\-1969&-1983&1183\end{pmatrix}.
\]
Both matrices are positive definite. The first two partial products lie in the following intervals:
\[
\begin{array}{c|cc}
 k&p_k(X)&p_k(Y)\\\hline
 1&(33624,33625)&(69915,69916)\\
 2&(264.23,264.25)&(242.47,242.48)
\end{array}
\]
Thus $X$ and $Y$ are incomparable under log-majorization. In both examples, replacing $A,B$ by $A^r,B^r$ and taking $s=1/r$ gives a counterexample to $Q_t(A^s,B^s)\logmaj Q_t(A,B)^s$.
\end{example}

\medskip
\noindent\textbf{Declaration of AI assistance.} OpenAI's ChatGPT was used for literature searches, exploratory and symbolic calculations, the development and examination of proof strategies, and language and \LaTeX{} editing. The authors are responsible for the verification of all mathematical arguments, references, and numerical results, and for the final content of the manuscript.


\begin{thebibliography}{99}
\bibitem{Araki} H. Araki, On an inequality of Lieb and Thirring, \emph{Lett. Math. Phys.} \textbf{19} (1990), 167--170.
\bibitem{Aud} K. M. R. Audenaert, In-betweenness, a geometrical monotonicity property for operator means, \emph{Linear Algebra Appl.} \textbf{438} (2013), 1769--1778.
\bibitem{BhatiaMA} R. Bhatia, \emph{Matrix Analysis}, Graduate Texts in Mathematics 169, Springer, New York, 1997.
\bibitem{BhatiaPD} R. Bhatia, \emph{Positive Definite Matrices}, Princeton University Press, Princeton, NJ, 2007.
\bibitem{BJL} R. Bhatia, T. Jain, and Y. Lim, Inequalities for the Wasserstein mean of positive definite matrices, \emph{Linear Algebra Appl.} \textbf{576} (2019), 108--123.
\bibitem{DinhBMV} T. H. Dinh, On the failure of the upper bound in the refined BMV conjecture and a pinching correction, \href{https://arxiv.org/abs/2605.17782}{arXiv:2605.17782} (2026).
\bibitem{DAT} T. H. Dinh, S. Ahsani, and T.-Y. Tam, Geometry and inequalities of geometric mean, \emph{Czechoslovak Math. J.} \textbf{66} (2016), 777--792.
\bibitem{DDF} T. H. Dinh, R. Dumitru, and J. A. Franco, Some geometric properties of matrix means with respect to different metrics, \emph{Positivity} \textbf{24} (2020), 1419--1434.
\bibitem{DLVBMV} T. H. Dinh, C. T. Le, and T. D. Vuong, Positive planar measures for spectral pinching and the corrected BMV conjecture, \href{https://arxiv.org/abs/2610.31851}{arXiv:2610.31851} (2026).
\bibitem{DTV} T. H. Dinh, T.-Y. Tam, and T.-D. Vuong, A new weighted spectral geometric mean and properties, \emph{Electron. J. Linear Algebra} \textbf{40} (2024), 333--342.
\bibitem{DFnear} R. Dumitru and J. A. Franco, Near order and metric-like functions on the cone of positive definite matrices, \emph{Positivity} \textbf{28} (2024), Paper No. 2.
\bibitem{DFrenyi} R. Dumitru and J. A. Franco, Log-majorizations related to the spectral geometric and R\'enyi means, \emph{Acta Sci. Math. (Szeged)} \textbf{90} (2024), 551--563.
\bibitem{FP} M. Fiedler and V. Pt\'ak, A new positive definite geometric mean of two positive definite matrices, \emph{Linear Algebra Appl.} \textbf{251} (1997), 1--20.
\bibitem{FGK} J. A. Franco, L. Gan, and S. Kim, Near-geometric mean of positive definite operators, \emph{Positivity} \textbf{29} (2025), Paper No. 54.
\bibitem{FS} S. Furuichi and Y. Seo, Some inequalities for spectral geometric mean with applications, \emph{Linear Multilinear Algebra} \textbf{73} (2025), 1508--1527.
\bibitem{Furuta} T. Furuta, $A\geq B\geq0$ assures $(B^r A^p B^r)^{1/q}\geq B^{(p+2r)/q}$ for $r\geq0$, $p\geq0$, $q\geq1$ with $(1+2r)q\geq p+2r$, \emph{Proc. Amer. Math. Soc.} \textbf{101} (1987), 85--88.
\bibitem{GH} L. Gan and H. Huang, Order relations of the Wasserstein mean and the spectral geometric mean, \emph{Electron. J. Linear Algebra} \textbf{40} (2024), 491--505.
\bibitem{GT} L. Gan and T.-Y. Tam, Inequalities and limits of weighted spectral geometric mean, \emph{Linear Multilinear Algebra} \textbf{72} (2024), 261--282.
\bibitem{Hiai} F. Hiai, Log-majorization and matrix norm inequalities with application to quantum information, \emph{Acta Sci. Math. (Szeged)} \textbf{90} (2024), 527--549.
\bibitem{JKT} M. Jeong, S. Kim, and T.-Y. Tam, New weighted spectral geometric mean and quantum divergence, \emph{Linear Algebra Appl.} \textbf{726} (2025), 164--179.
\bibitem{KA} F. Kubo and T. Ando, Means of positive linear operators, \emph{Math. Ann.} \textbf{246} (1980), 205--224.
\bibitem{LL} H. Lee and Y. Lim, Metric and spectral geometric means on symmetric cones, \emph{Kyungpook Math. J.} \textbf{47} (2007), 133--150.
\bibitem{PW} W. Pusz and S. L. Woronowicz, Functional calculus for sesquilinear forms and the purification map, \emph{Rep. Math. Phys.} \textbf{8} (1975), 159--170.
\bibitem{SS} M. Shaked and J. G. Shanthikumar, \emph{Stochastic Orders}, Springer, New York, 2007.
\end{thebibliography}
\end{document}